\documentclass{amsart}
\usepackage{amsfonts,amssymb,amsmath,amsthm}
\usepackage{url}
\usepackage{enumerate}

\newtheorem{thrm}{Theorem}[section]
\newtheorem{lem}[thrm]{Lemma}
\newtheorem{prop}[thrm]{Proposition}

\theoremstyle{definition}
\newtheorem{definition}[thrm]{Definition}
\newtheorem{discussion}[thrm]{Discussion}
\newtheorem{question}[thrm]{Question}
\newtheorem{remark}[thrm]{Remark}
\newtheorem{example}[thrm]{Example}
\numberwithin{equation}{section}

\newcommand{\Ass}{\operatorname{Ass}}

\newcommand{\fgrade}{\operatorname{fgrade}}

\newcommand{\Spec}{\operatorname{Spec}}

\newcommand{\rad}{\operatorname{rad}}
\newcommand{\cd}{\operatorname{cd}}

\newcommand{\Ht}{\operatorname{ht}}
\newcommand{\pd}{\operatorname{pd}}

\newcommand{\Ext}{\operatorname{Ext}}

\newcommand{\Hom}{\operatorname{Hom}}

\newcommand{\Ann}{\operatorname{Ann}}
\newcommand{\rank}{\operatorname{rank}}

\newcommand{\depth}{\operatorname{depth}}

\newcommand{\vpl}{\operatornamewithlimits{\varprojlim}}

\newcommand{\fm}{\frak{m}}
\newcommand{\fp}{\frak{p}}

\newcommand{\fn}{\frak{n}}

\author{ Kh. Ahmadi-Amoli\ \ }

\author{ E. Banisaeed\ \ }
\author{ M. Eghbali\ \ }
\author{ F. Rahmati}

\address{Department of Mathematics , Payame noor University, Tehran, 19395--3697, Iran.}
 \email{khahmadi@pnu.ac.ir}

 \address{Department of Mathematics , Payame noor University, Tehran, 19395--3697, Iran.}
 \email{banisaeed.elahe@yahoo.com}

\address{School of Mathematics, Institute for Research in Fundamental Sciences (IPM), P. O. Box: 19395-5746,
Tehran-Iran.} \email{m.eghbali@yahoo.com}

\address{Department of Pure Mathematics , Faculty of Mathematics and Computer Science, Amirkabir University of technology (Tehran Polytechnic), 424 Hafez Ave., 15914, Tehran, Iran.}
 \email{frahmati@aut.ac.ir}
\thanks{The third named author's research was in part supported by a grant from IPM (No. 94130017) }


\begin{document}

\title[formal grade, depth and Lubeznik numbers]{ On the relation between formal grade and depth with a view toward vanishing of Lyubeznik numbers}

\begin{abstract}
Let $(R,\fm)$ be a local ring and $I$ an ideal. The aim of the present paper is twofold.
At first we continue the
investigation to compare $\fgrade(I,R)$ with
$\depth R/I$ and further we derive some results
on the vanishing of Lyubeznik numbers.
\end{abstract}
 \maketitle

\section{Introduction} \label{sect1}

Let $(R,\fm)$ be a commutative Noetherian local ring and $I$ be an
ideal of $R$. For an $R$-module $M$, we denote by $H^i_I(M)$ the
$i$th local cohomology module of $M$ with respect to $I$. For more
details the reader may consult \cite{Br-Sh}.

Consider the family of local cohomology modules $\{H^i_{\fm}
(M/I^nM)\}_{n \in \mathbb{N}}$. For every $n$ there is a natural
homomorphism $H^i_{\fm} (M/I^{n+1}M) \rightarrow H^i_{\fm} (M/I^nM)$
such that the family forms a projective system. The projective limit
${\vpl}_nH^i_{\fm}(M/I^n M)$ is called the $i$th formal local
cohomology of $M$ with respect to $I$ (cf. \cite{Sch}). For more information
 on this topic we refer the reader to \cite{Sch}, \cite{As-D} and \cite{E13}.
In the case $(R,\fm)$ is a Gorenstein local ring Peskine and Szpioro \cite{Pe-Sz} proved the isomorphism
$$H^i_{p}(\mathcal{X}, \mathcal{O}_{\mathcal{X}}) \cong {\vpl}_nH^i_{\fm}(R/I^n),$$
where $\mathcal{X}$ is the formal completion of $\Spec R$ along $\Spec R/I$ and $p$ a closed point.

 The formal
grade, $\fgrade(I,M)$, is defined as the index of the minimal
nonvanishing formal cohomology module, i.e., $\fgrade(I,M) = \inf
\{i \in \mathbb{Z}|\ \ {\vpl}_nH^i_{\fm}(M/I^n M) \neq 0\}$.

The concept of formal grade
% is not only interesting for the mentioned reason but also
 can be used to
investigate vanishing conditions of local cohomology modules. One
way to check out vanishing of local cohomology modules is the
following duality
\begin{equation}\label{duality}
    {\vpl}_nH^i_{\fm}(R/I^n) \cong \Hom_R(H^{\dim R-i}_{I}(R),E(R/\fm)),
\end{equation}
where $(R,\fm)$ is a Gorenstein local ring and $E(R/\fm)$ denotes
the injective hull of the residue field (cf. \cite[Remark
3.6]{Sch}).

 In the case that $(R,\fm)$ is a regular local ring containing a field $k$ the following consideration of formal grade
have been given:
\begin{itemize}
  \item[(1)] If $k=\mathbb{Q}$, then $\fgrade(I,R) \leq$ de Rham depth of $\Spec R/I$, (cf. \cite[Theorem 2.8]{Og}).
  \item[(2)] If $k$ is of positive characteristic and $R$ is $F$-finite, then  $\fgrade(I,R) =$ Frobenius depth of $R/I$, (cf. \cite[Corollary 3.9]{E15}).
\end{itemize}

From the other viewpoint, the computation of $\depth(R/I^n)$, even for small values of $n$,
can be very difficult. So, determining of $\min_n \depth(R/I^n)$ in
most of the cases is abortive. A natural upper bound for the $\min_n
\depth(R/I^n)$ is the $\depth R/I$. In an attempt to find a suitable
(homological) upper bound for the $\min_n \depth(R/I^n)$, the third
named author in \cite{E}, examined the formal grade of $I$,
$\fgrade(I,R)$ and gave concrete examples for the equality $\min_n
\depth(R/I^n) = \fgrade(I,R)$.
If $(R,\fm)$ is a regular local ring, in the light of
equality (\ref{duality}) and \cite[Lemma 4.8]{Sch} one has
$$\fgrade(I,R)=\depth R/I \ \ \ \text{\ if\ and\ only\ if\ }\ \ \ \pd R/I=\cd(R,I).$$

We let $\pd(-)$ resp. $\cd(-,-)$ to be the projective dimension resp.
cohomological dimension (cf. Section 2).

The relation between $\fgrade(I,R)$ and $\depth R/I$ is not clear in
general. It is known that $\fgrade(I,R) \leq \dim R/I$, (cf.
\cite[Lemma 4.8]{Sch}). The equality $\fgrade(I,R) = \depth R/I$ is
achieved in the following cases:
\begin{itemize}
  \item $R$ is a regular local ring of positive characteristic and
  $R/I$ is Cohen-Macaulay, (cf. \cite[Remark 3.1]{E}).
  \item $R=k[x_1, \ldots, x_n]_{(x_1, \ldots, x_n)}$ and $I$ is a
  squarefree monomial ideal, (cf. \cite[Corollary 4.2]{E}).
\end{itemize}

The equality $\fgrade(I,R) = \depth R/I$ is not the case for non
Cohen-Macaulay ideals in general. Let $I = (x_0x_3-x_1x_2,
x^3_1-x^2_0x_3, x^3_2-x_1x^2_3) \subset R=k[x_0, x_1,x_2,x_3]$ be
the defining ideal of the twisted quartic curve in ${\mathbb P}^3$.
Using Mayer-Vietoris sequence one can see that $\depth R/I =1$ and
$\fgrade(I,R)=2$.

We conclude this introduction by giving a brief preview of what is to come.
Our main contribution in Section 2, is to further investigate on the equality
between $\fgrade(I,R)$ and $\depth R/I$. More precisely, we prove the following.

\begin{thrm}  (cf. Theorem \ref{Main}) Suppose that $I$ is a equimultiple ideal of a regular
local ring $(R,\fm)$ containing a field of characteristic zero. Then,
$$\fgrade(I,R) = \depth R/I^n = \depth R/I^{(n)},\ \text{\ for\ all\ } n > 0,$$
where $I^{(n)}$ denote the $n$th symbolic power of $I$.
\end{thrm}

As a consequence of the relationship between $\fgrade(I,R)$ and
$\depth R/I$, in Section 3, we derive some results on the vanishing
of Lyubeznik numbers (cf. Section $3$ for the definitions).
Lyubeznik numbers carry some geometrical and topological information
as indicated in \cite{Lyu2}. For more applications of these
invariants we refer the reader to \cite{GL-S}, \cite{Bl-Bo},
\cite{Lyu3}, \cite{Bl}, \cite{Zh} and \cite{Al-Ya}.

In \cite{V} Varbaro raised the following question:

\begin{question} \label{Question} Suppose that $R$ is a regular local ring containing a field and $I$ an
ideal. For all $j < \depth R/I$ and $i \geq j -2$ is it true that
the Lyubeznik numbers $\lambda_{i,j}(R/I) = 0?$
\end{question}

We give some partial positive answers to the aforementioned question as follows.

\begin{thrm} (cf. Theorem \ref{flatendomorphism} and Proposition \ref{last}) Let $(R,\fm)$ be a local ring  and let $I$ be an ideal of $R$. Assume that one of the following statements holds.
\begin{itemize}
  \item[(1)] Suppose that projective dimension of $R/I$ is finite and $\varphi: R \rightarrow R$ is a ring endomomorphism satisfying going down property such that $\{\varphi^t(I)R\}_{t \geq 0}$ is a decreasing chain of ideals cofinal with
$\{I^t\}_{t \geq 0}$ and  $\depth R/I = r$.
  \item[(2)] Let $R$ be a regular complete ring containing a
field and
let $I$ be a Cohen-Macaulay ideal of $R$ which is generically complete intersection.
Further assume that $\Ht(I) = l(I) - 1$ and $J$ is a proper minimal reduction of $I$
with $\depth R/J=r$.
\end{itemize}
  Then we have
$$\lambda_{i-t,i}(R/I) = 0\ \text{\ for\ all\ } 0 \leq i <r \text{\ and\ } 0 \leq t \leq i.$$
\end{thrm}

It is noteworthy to mention that in the case $R$ is essentially of
finite type the Question \ref{Question} has been proved in
\cite{Da-Ta}.

\section{Formal grade and depth}

Let $(R,\fm)$ be a commutative Noetherian local ring and $I$ be an
ideal of $R$. In this section we will further investigate on the
relationship between $\depth R/I$ and $\fgrade(I,R)$. First of all
we review some notions and results of formal grade and related
topics.

\begin{itemize}
  \item [(2-1)] $\dim M/IM = \sup \{i \in \mathbb{Z}|\ \ \  {\vpl}_nH^i_{\fm}(M/I^n M)
\neq 0\}$, (cf. \cite[Theorem 4.5]{Sch}).
  \item [(2-2)] $ \fgrade(I,M) \leq \min \{ \depth_R M,\dim M/IM\}$, (cf. \cite[Lemma 4.8]{Sch}).
 \item [(2-3)] Suppose that $R$ is a Cohen-Macaulay local ring. Then $\fgrade(I,R)+ \cd(R,I) = \dim
 R$, (cf. \cite{As-D}), where $\cd(R,I) =\sup \{i \in \mathbb{Z} \mid \ H^i_I(R) \neq 0 \}$.
  \item [(2-4)] $\min_n \depth R/I^n \leq \fgrade(I,R)$, (cf. \cite[Proposition 2.2]{E}).
\end{itemize}

An ideal $I$ is called cohomologically complete intersection if
$\Ht(I)=\cd(R,I)$, (cf. \cite{H-Sch}). Suppose the residue field
$R/\fm$ is infinite. We denote by $l(I)$ the Krull dimension of
$\oplus^{\infty}_{n=0}(I^n/I^n \fm)$, called the analytic spread of
$I$. It is known that $\Ht(I) \leq l(I) \leq \mu(I)$, where $\mu(I)$
is the minimal number of generators of $I$. The ideal $I$ is called
equimultiple if $\Ht(I) = l(I)$. There are large classes of such
ideals, for instance, every $\fm$-primary ideal $I$ is equimultiple
 (cf. \cite[Theorem 1, page 154]{No-Re}).

\begin{remark}  Let $(R,\fm)$ be a Cohen-Macaulay local ring with infinite residue field and
 $I$ be a non Cohen-Macaulay cohomologically complete intersection ideal.
 Then
 \begin{enumerate}
   \item [(a)]  $\depth R/I \leq \fgrade(I,R)-1$.
   \item [(b)] If $l(I)-\Ht(I)=1$, $\Ht(I) \geq 1$ and the graded associated
   ring $gr_I(R)$ is Cohen-Macaulay, then $\depth R/I \geq \fgrade(I,R)-1$.
 \end{enumerate}
\end{remark}

\begin{proof} (a) As $I$ is a cohomologically complete intersection ideal
which is not Cohen-Macaulay, $\depth R/I \leq \dim R- \cd(I,R)-1$.
The later is nothing but  $\fgrade(I,R)-1$, by the equality (2-3)
above.

(b) Since $gr_I(R)$ is Cohen-Macaulay, by \cite[Proposition
3.3]{Ei-Hu} we have $l(I)=\dim R- \min \depth R/I^n$. Now the
assertion follows by the following inequality
$$1 \geq \dim R- \depth R/I-\cd(R,I) =
\fgrade(I,R)-\depth R/I.$$
\end{proof}

An ideal  $J \subset I$ is a reduction of
 $I$ if $JI^r = I^{r+1}$ for some integer $r$. The least
such integer is the reduction number of $I$ with respect to $J$.

\begin{remark}  Let $(R,\fm)$ be a Cohen-Macaulay local ring with infinite residue field. Assume that
 $I$ is a Cohen-Macaulay  ideal with $l(I)-\Ht(I)=1$ and $\Ht(I) \geq 1$. Further assume that the reduction number of $I$ is $1$.
 Then
$$\depth R/I \geq \fgrade(I,R) \geq \depth R/I-1.$$
\end{remark}

\begin{proof}
By \cite[Corollary 5.3.7]{Vas},  $gr_I(R)$ is Cohen-Macaulay. Once again using
\cite[Proposition 3.3]{Ei-Hu}, the claim is clear.
\end{proof}

As hinted at in the introduction, for a regular local ring
$(R,\fm)$  containing a field of positive characteristic,
$\fgrade(I,R)=\depth R/I$, whenever $I$ is a Cohen-Macaualy ideal
(\cite[Remark 3.1]{E}). The following result (Theorem \ref{Main}) can be thought of as
its analogue in characteristic zero.

\begin{definition} Let $R$ be a reduced ring and $I$ be an ideal. If $\fp_1,\ldots, \fp_n$ are minimal
primes of $I$, we set $S := R \setminus ( \fp_1 \cup \ldots \cup
\fp_n)$ a multiplicatively closed subset of $R$. The $n$th symbolic
power of $I$ with respect to $S$ is defined as $I^{(n)} =
I^n(S^{-1}R) \cap R$.
\end{definition}

\begin{thrm} \label{Main} Suppose that $I$ is a equimultiple
 ideal of a regular local ring $(R,\fm)$ containing a field
of characteristic zero. Then,
$$\fgrade(I,R)=\depth R/I^n=\depth R/I^{(n)},\ \text{\ for\ all\ }
n>0.$$
\end{thrm}

\begin{proof} Since $l(\rad(I))=l(I)$,  $\fgrade(\rad(I),R)=\fgrade(I,R)$,
$\depth R/(\rad(I))^n=\depth R/I^n$ and $\depth
R/(\rad(I))^{(n)}=\depth R/I^{(n)}$, then we may assume that $I$ is
a radical ideal. On the other hand, it is known that $l(I)=l(I
\hat{R})$ and since $R$ is regular, for all integer $n$ one has
$\hat{I}^{(n)} = I^{(n)}\hat{R}$ (see for instance \cite[page
14]{Hun-Hoc}) where $\hat{R}$ is the completion of $R$ with respect
to the maximal ideal $\fm$. Moreover, depth and formal grade are
stable under completion, (cf. \cite[Proposition 3.3]{Sch}). Thus, by
passing to the completion we may assume that $R$ is a complete
regular local ring containing a field of characteristic zero. Hence,
$R$ contains a field of characteristic zero that maps onto its
residue field. It implies that $R/\fm$ is infinite. Now, by our
assumptions, the celebrated result of Cowsik and Nori \cite{C-N},
implies that $I$ is a complete intersection ideal. Thus, $R/I^n$ is
Cohen-Macaulay for all $n>0$, then so is $R/I^{(n)}$
 for all $n>0$, by \cite[Corollary 3.8]{Li-Sw}. Hence, $H^i_{\fm}(R/I^n)=0$ for all
  $i < \depth R/I^n=\dim R/I$. On the other hand, one has
  ${\vpl}_nH^i_{\fm}(R/I^n)=0$ for all
  $i <\dim R/I$. Accordingly, by virtue of equality (2-1) we get
$$\fgrade(I,R)=\depth R/I^n=\depth R/I^{(n)}, \  \text{\ for\ all\ } n>0.$$
\end{proof}

In the Theorem \ref{Main} we can not remove the assumption
$\Ht(I)=l(I)$. Consider the following example.

\begin{example}  Let $R=\mathbb{C}[x,y,z,w]_{(x,y,z,w)}$ and $I=(x,y) \cap
(z,w)$ be an ideal of $R$. We have $\Ht(I)=2$, $l(I)=3$, where $I$
is a radical ideal. On the other hand,  $\fgrade(I,R)=2$ and $\dim
R/I=2$ but
  for all positive integers $n$, one has
$I^n=(x,y)^n \cap (z,w)^n$. Hence, $ \depth R/I^n=1$ for all $n>0$.
\end{example}

\begin{prop}  Let $(R,\fm)$ be a local ring of positive dimension
and $I$ a non-Cohen-Macaulay ideal of $R$. If $\depth R/I>0$ then
$\fgrade(I,R)>0$.
\end{prop}

\begin{proof} Since depth and formal grade are not effected under completion,
without loss of generality we may assume that $R$ is a
complete local ring.

In the contrary, suppose that $\fgrade(I,R)=0$. Thus,
${\vpl}_nH^0_{\fm}(R/I^n)\neq 0$. By virtue of \cite[Lemma
4.1]{Sch}, there exists $\fp \in \Ass R$ such that $I+\fp$ is an
$\fm$-primary ideal. So, by The Independence Theorem, one has
\begin{equation}\label{A}
   H^0_{\fp}(R/I) \cong H^0_{\fm}(R/I) \cong \bigcup_t(I:_R \fp^t),
\end{equation}
where the later module is zero by the assumption. Since $\fp \in
\Ass R$, there exists a non zero element $x \in R$ such that
$\fp=\Ann (xR)$. It yields that $0 \neq x \in \bigcup_t(I:_R
\fp^t)$, contradicts to the vanishing in (\ref{A}).
\end{proof}

\begin{remark}  The previous result is not true for Cohen-Macaulay
ideals:

Put $R=k[[x,y,z]]/(xy,xz)$ and $I=(x,y)R$. Note that $(y,z)R \in
\Ass R$, $(y,z)R+I=(x,y,z)R$. Hence, ${\vpl}_nH^0_{\fm}(R/I^n)\neq
0$ but $\depth R/I=1$.
\end{remark}

\begin{discussion} Let $(R,\fm)$ be a Cohen-Macaulay local ring
and $I$ an ideal of $R$. If $\depth R/I>t>1$ is it true that
$\fgrade(I,R)>t$?  To answer this question in
special cases if either $I$ is a cohomologically complete intersection ideal or $I$ is a
non-Cohen-Macaulay ideal with $\Ht(I) = \cd(R,I) -1$, by virtue of (2-3), one has
$$\fgrade(I,R) \geq \depth R/I.$$
It should be mentioned that for a regular local ring $(R,\fm)$
containing a field, it has been proved in \cite[Proposition 3.1]{V}
that if $\depth(R/I)> 1$, then $\fgrade(I,R)> 1$.
\end{discussion}

\section{Lyubeznik numbers}

Let $R$ be a local ring which admits a surjection from an
$n$-dimensional regular local ring $S$ containing a field. Let $I$
be the kernel of the surjection, and $k=S/\fm$. Lyubeznik
\cite{Lyu2} proved that the Bass numbers $\lambda_{i,j}(R) = \dim_k
\Ext^i_ S(k,H^{n-j}_I(S))$, known as Lyubeznik numbers of $R$,
depend only on $R, i$ and $j$, but neither on $S$ nor on the
surjection $S \rightarrow R$. Note that, in the light of \cite{Hun-Sh} and \cite{Lyu2} these Bass
numbers are all finite.

Put $d = \dim R$, Lyubeznik numbers satisfy the following properties:
\begin{enumerate}
  \item [(a)] $\lambda_{i,j}(R) = 0$ for $j > d$ or $i > j$.
  \item [(b)] $\lambda_{d,d}(R)\neq 0$.
  \item [(c)] If $R$ is Cohen-Macaulay, then $\lambda_{d,d}(R) = 1$ (cf. \cite{kaw}).
  \item [(d)] Euler characteristic (cf. \cite{Al}):
  $$\sum_{0 \leq i,j \leq d} (-1)^{i-j}\lambda_{i,j}(R)=1.$$
\end{enumerate}
Therefore, we can record all nonzero Lyubeznik numbers in the so-called Lyubeznik
table:

$$\Lambda(R)=\left(
         \begin{array}{cccccc}
           \lambda_{0,0} &   \lambda_{0,1}  & \ldots &      \lambda_{0,d-1} &   \lambda_{0,d}    \\
       0 &   \lambda_{1,1} & \ldots &     \lambda_{1,d-1} &   \lambda_{1,d}    \\
        0 &   0  & \ddots &  \vdots &   \vdots    \\
     %   \vdots &   \vdots  & \ddots &      \lambda_{d-2,d-1} &   \lambda_{d-2,d}    \\
        0 & 0 & \ldots &      \lambda_{d-1,d-1} &   \lambda_{d-1,d}    \\
        0 &  0 & \ldots &    0 &   \lambda_{d,d}
      \end{array}
      \right),$$

where $\lambda_{i,j}:=\lambda_{i,j}(R) $ for every $0 \leq i,j \leq d$.

\begin{prop}  \label{Varbaro} (\cite[Corollary 3.3]{V}) Let $R$ be as above. If $\depth (R)=r$, then
 $$\lambda_{i-1,i}(R) = \lambda_{i,i}(R)=0$$
 for all $0 \leq i < r$.
\end{prop}

In this section, in order to give a positive answer to the Question
\ref{Question} we try to give a slight extension of the Proposition
\ref{Varbaro}.

\begin{lem}  \label{Lyunumb} Let $(R,\fm)$ be a local ring
admits an epimorphism  from a regular local ring $(S,\fn)$ which is
containing a field and let $I$ be its kernel. Assume that
$\fgrade(I,S) \geq \depth S/I
= r,\ r \in \mathbb{N},$ then
 $$\lambda_{i-t,i}(R) = 0 \text{\ \ for\ all\ \
} 0 \leq i < r \text{\ \ and\ \ } 0 \leq t \leq i.$$
\end{lem}

\begin{proof} The proof follows from the equality (\ref{duality}) mentioned
in the Introduction.
\end{proof}

As a byproduct of the Lemma \ref{Lyunumb}, we may prove the Theorem \ref{flatendomorphism}.
 Before it we need a lemma pointed out to the authors by Matteo Varbaro.
 It provides a slight generalization of our earlier result. Let us recall
 that for two rings $R$ and $S$, a ring map $\varphi:R \rightarrow S$ has
  the going-down property if, for any two primes $\fp' \subsetneq \fp$ of
   $R$ and any prime $\mathcal{P}$ in $S$ with $\varphi^{-1}(\mathcal{P})=\fp$,
    there is a prime $\mathcal{P'} \subsetneq \mathcal{P}$ of $S$ such that
    $\varphi^{-1}(\mathcal{P'})=\fp'$.

\begin{remark} Notice that every flat homomorphism implies the going
 down property. Take into account that the Frobenius endomorphism
  of a ring of positive characteristic and,
for any field $k$ and any integer $t \geq 2$, the $k$-linear
 endomorphism $\varphi(x_i) = x^t_i$
of $k[x_1,\ldots, x_n]$ are the prototypical examples of flat
endomorphisms.
\end{remark}

\begin{lem} \label{Matteo}
Let $R$ be a Noetherian local ring and $I \subset R$ an ideal of finite projective
 dimension. Assume that $\varphi:R \rightarrow R$ is a ring  endomomorphism
 satisfying the going down property and $J=\varphi(I)R$. Then $\depth R/I \leq \depth R/J$.
\end{lem}

\begin{proof} Suppose that $\pd R/I=n$.
 Let $F_n \rightarrow F_{n-1} \rightarrow \ldots \rightarrow F_0$ be a free resolution
  of $R/I$, with maps given by matrices $A_i=(a_{pq}):F_i \rightarrow F_{i-1}$,
   where $p, q$ are integers.
   Applying  the map $\varphi$ to this free resolution, we get a complex of free
    modules with maps given by $\varphi(A_i)=(\varphi(a_{pq}))$. By the acyclicity
     criterion of Buchsbaum-Eisenbud (cf. \cite[Theorem 1.4.13]{Br-Her}),
      the ideal of minors of the matrices $A_i$ namely $I_{r_i}(A_i)$ have
        grade $\geq i$ for $i=1, \ldots , n$, where
         $r_i=\sum^n_{j=i} (-1)^{j-i} \rank F_j$. But, since $\varphi$ satisfies
         the going down property, the ideals of minors of the $\varphi(A_i)$, which
          are just the ideal of minors of $A_i$ extended by $\varphi$, have enough
           grade too. Once again using the Buchsbaum-Eisenbud criterion one
            has $\pd R/J \leq \pd R/I$. Now, the assertion follows by the
             Auslander-Buchsbaum Theorem.
\end{proof}

\begin{thrm} \label{flatendomorphism}  Let $(R,\fm)$ be a regular local ring
containing a field and let $I$ be an ideal of $R$. Assume that
$\varphi:R \rightarrow R$ is a ring endomomorphism satisfying going
down property. Further assume that $\{\varphi^t(I)R\}_{t\geq 0}$ is
a decreasing chain of ideals cofinal with $\{I^t\}_{t\geq 0}$. If
$\depth R/I = r$ then
 $$\lambda_{i-t,i}(R/I) = 0 \text{\ \ for\ all\ \
} 0 \leq i < r \text{\ \ and\ \ } 0 \leq t \leq i.$$
\end{thrm}

\begin{proof} It is clear that composition of maps satisfying the going down
 property, satisfies  the going down property itself.
 In particular, if $\varphi$ has the going-down property, the so does
 $\varphi^t$ for any positive integer $t$. By virtue of Lemma \ref{Matteo}, one has
$$\depth R/I \leq \depth R/\varphi(I)R \leq \depth R/\varphi^2(I)R \leq \depth R/\varphi^3(I)R \leq \ldots.$$
Hence, $H^i_{\fm}(R/\varphi^t(I)R)=0$ for all $i < r$ and all integer $t$. By applying inverse limit one has $0={\vpl}_tH^i_{\fm}(R/\varphi^t(I)R)\cong {\vpl}_tH^i_{\fm}(R/I^t)$ for all $i < r$, i.e $\fgrade(I,R)\geq \depth R/I$. To this end, note
that the isomorphism follows from \cite[Lemma 3.8]{Sch}. Now we are done by Lemma \ref{Lyunumb}.
\end{proof}

With the assumptions of Theorem \ref{flatendomorphism} we see that all the entries from column
$1$ up to column $r$ in the Lyubeznik table of $R/I$ are zero.

We conclude this section with a result on the vanishing of Lyubeznik
numbers using the concept of minimal reduction. A reduction $J$ of
$I$ is minimal if it does not contain properly another reduction of
$I$.

\begin{prop} \label{last} Let $(R,\fm)$ be a complete regular local ring containing a field and
let $I$ be a Cohen-Macaulay ideal of $R$ which is generically complete intersection.
Further assume that $\Ht(I) = l(I) - 1$ and $J$ is a proper minimal reduction of $I$.
Then
 $$\lambda_{i-t,i}(R/I) = 0 \text{\ \ for\ all\ } \ 0 \leq i < r \text{\ \ and\ \ } 0 \leq t \leq i,$$
 where $\depth R/J=r$.
\end{prop}

\begin{proof}
Once again using Lemma \ref{Lyunumb}, it is enough to prove that $\fgrade(I,R) \geq
\depth R/J$. It is known that $\rad(I) = \rad(J)$, thus we have $\fgrade(I,R) = \fgrade(J,R)$.

Suppose $k = R/\fm$. If $char(k) > 0$ the result is known (cf.
\cite[Remark 3.1]{E}), so we can assume $char(k) = 0$ ($k$ is
infinite). By the assumptions we have
$$\depth R/J = \dim R- \Ht(J)- 1 = \dim R- l(I) \leq \dim R- \cd(R,I) =\fgrade(I,R),$$
where the first equality follows from \cite[Theorem 5.3.4]{Vas} and
the last one is clear from (2-3). To this end note that $\cd(R,I)
=\cd(R,J) \leq \mu(J)=l(I)$.
\end{proof}

\proof[Acknowledgment]

The authors are grateful to Josep  \`{A}lvarez Montaner and Matteo Varbaro for careful reading of the earlier draft of this paper and their useful comments make this paper readable.

\end{document}